\documentclass[a4paper,11pt]{amsart}

\usepackage{cite}
\usepackage{url}

\usepackage{hyperref}

\usepackage{enumerate}

\newtheorem{theorem}{Theorem}[section]

\numberwithin{equation}{section}

\newcommand{\Fcal} {{\mathcal F}}

\newcommand{\Ocal} {{\mathcal O}}

\newcommand{\R}{\mathbb{R}}
\newcommand{\N}{\mathbb{N}}

\renewcommand{\P}{\mathbb{P}}
\newcommand{\E}{\mathbb{E}}

\renewcommand{\epsilon}{\varepsilon}
\newcommand{\eps}{\varepsilon} 

\renewcommand{\div}{\operatorname{div}}

\usepackage{geometry}

\title{Improved regularity for the stochastic fast diffusion equation}

\author{Ioana Ciotir}
\address{Normandie University\\
INSA de Rouen Normandie\\
LMI (EA 3226 - FR CNRS 3335)\\
76000 Rouen\\
France}
\email{ioana.ciotir@insa-rouen.fr}
\author{Dan Goreac}
\address{School of Mathematics and Statistics\\
Shandong University\\
Weihai\\
Weihai 264209\\
China
\& LAMA\\
Univ Gustave Eiffel, UPEM\\
Univ Paris Est Creteil, CNRS\\
77447 Marne-la-Vall\'ee\\
France}
\email{dan.goreac@univ-eiffel.fr}
\author{Jonas M. T\"olle}
\address{Aalto University\\
Department of Mathematics and Systems Analysis\\
PO Box 11100 (Otakaari 1, Espoo)\\
00076 Aalto\\
Finland}
\email{jonas.tolle@aalto.fi}

\date{\today}

\begin{document}

\begin{abstract}
We prove that the solution to the singular-degenerate stochastic fast-diffusion equation with parameter $m\in (0,1)$, with zero Dirichlet boundary conditions on a bounded domain in any spatial dimension, and driven by linear multiplicative Wiener noise, exhibits improved regularity in the Sobolev space $W^{1,m+1}_0$ for initial data in $L^{2}$.\end{abstract}


\thanks{IC was partially supported by l'Agence Nationale de la Recherche (ANR), project ANR COSS number ANR-22-CE40-0010. DG has been partially supported by the National Key R and D Program of China (No. 2018YFA0703901) and the NSF of China (Nos. 12031009, 11871037). The research of JMT was partially supported by the seed funding grant UNA Random of the Una Europa alliance.}


\maketitle

\section{Introduction}

In this work, we establish higher order regularity of the strong solutions to
the stochastic fast diffusion equation perturbed by linear multiplicative Wiener noise.
The equations are set on a bounded domain $\mathcal{O}\subset\R^d$ with sufficiently smooth boundary, and formulated with zero Dirichlet boundary conditions. Our approach is independent of the space dimension.

The singular-degenerate stochastic fast diffusion equation, $m\in (0,1)$, until the time-horizon $T>0$, is given by
\begin{equation}\label{eq:main}
\left\{
\begin{aligned}
du(t)&=\Delta \left( u^{[m]}(t)\right)\, dt+\sum_{k=1}^\infty g_k u(t)\, d\beta_k(t), &&  t\in (0,T],\quad\text{in}\;\Ocal,  \\ 
u(t)&=0, && t\in (0,T],\quad\text{on}\;\partial\Ocal,   \\ 
u(0) &=u_{0}, && \text{in}\;\Ocal,
\end{aligned}\right.
\end{equation}
where we employ the notation $x^{[m]}:=|x|^{m-1}x$, $x\in\R$, $m\in(0,1)$. The stochastic driving term is given by an independent family of standard one-dimensional Brownian motions $\{\beta_k(t)\}_{t\ge 0}$, $k\in\N$ supported by a filtered probability space $(\Omega,\Fcal,\{\Fcal_t\}_{t\ge 0},\P)$ satisfying the usual assumptions of completeness and right-continuity.
The noise coefficients $g_k$, $k\in\N$, are assumed to satisfy
\begin{equation}\label{eq:coeff}\sum_{k=1}^\infty \|g_k\|^2_{C^1(\overline{\Ocal})}=:C_g<\infty.\end{equation}
Denote $H:=H^{-1}_0(\Ocal)$, that is, the topological dual space of $H_0^1(\Ocal)=W^{1,2}_0(\Ocal)$. Furthermore, denote the $L^2(\Ocal)$-norm by $\|\cdot\|_2$ and the $H^{-1}_0(\Ocal)$-norm by $\|\cdot\|_H$. For $v\in H$, we introduce the following notation for the noise coefficient,
\[B(v)(h)=\sum_{k=1}^\infty g_k v(e_k,h)_H,\quad h\in H,\]
where $e_k\in H$, $k\in\N$ are the elements of an orthonormal basis of $H$. Then
$B:H\to L_2(H,H)$ is Lipschitz continuous, i.e.,
\[\|B(x)-B(y)\|_{L_2(H,H)}^2\le C_g  \|x-y\|_H^2,\quad x,y\in H,\]
see \cite[Section 3]{GR:2015}. Here, $L_2(H,H)$ denotes the space of linear Hilbert-Schmidt operators from $H$ to $H$. We also obtain
\[\|B(x)\|_{L_2(H,L^2(\Ocal))}^2\le \sum_{k=1}^\infty \|g_k\|^2_{C^0(\overline{\Ocal})}\|x\|^2_{L^{2}(\Ocal)},\quad x\in L^{2}(\Ocal).\]
The stochastic fast diffusion equation is closely related to the stochastic porous medium equation, see \cite{BDPR:16} and the references therein. Several properties of the solutions to stochastic fast diffusion equations have been studied, for instance, finite time extinction \cite{BDPR:12-2,G:15}, random attractors \cite{G:13}, invariance of subspaces \cite{Liu:10}, ergodicity and uniqueness of invariant measures \cite{GT:14,LT:11,N:23,BDP:10}, convergence of solutions \cite{CT:12,GT:16}, under general pseudodifferential operators \cite{RRW:07}, and regularity \cite{GR:2015}. 
The limiting case $m=0$ exhibits two particular frameworks, depending on how one interprets the passage to the limit for $m\to 0$. The multivalued case with a step-function nonlinearity is related to models of self-organized criticality and has been first studied in \cite{BR:09,BR:12,B:13,G:15} and is still an active topic of research \cite{BGN:22,N:21}. The logarithmic diffusion case has been studied in \cite{B:12,CFG:23}. The case $m\in (-1,0)$ is treated in \cite{C:17}.

For an initial datum $u_0\in L^2(\Omega,\Fcal_0,\P; H)$ and all spatial dimensions $d\in\N$, it is known that there exists a unique solution $\{u(t)\}_{t\in [0,T]}$ in the sense of stochastic variational inequalities (SVI) \cite[Definition 2.1]{GR:2015} to \eqref{eq:main} in the space $L^2(\Omega;C([0,T],H))$ by \cite[Theorem 2.3]{GR:2015}, which is also a unique generalized strong solution in the sense of \cite[Definition A.1]{GR:2015} by \cite[Theorem 3.1]{GR:2015}. At the same place, for initial data $u_0\in L^{m+1}(\Omega,\Fcal_0,\P;L^{m+1}(\Ocal))\cap L^2(\Omega,\Fcal_0,\P;H)$, the authors obtain that $u$ is in fact a unique pathwise strong solution in the sense of \cite[Definition A.1]{GR:2015}, such that
\begin{equation}\label{eq:Lm1}
u\in C([0,T]; L^{m+1}(\Omega;L^{m+1}(\Ocal))).
\end{equation}
Stronger notions of solutions and non-negativity of solutions are discussed in \cite[Section 3.6]{BDPR:16}, where the authors obtain $u^{[m]}\in L^2([0,T];H^1_0(\Ocal))$ and $\frac{d}{dt}u\in L^2([0,T];H)$ for $d=1,2,3$, where $m\in[\frac{1}{5},1]$ if $d=3$.

Our main result is given as follows.
\begin{theorem}\label{thm:main}
Assume that \eqref{eq:coeff} holds.
Then the unique strong solution $u$ to equation \eqref{eq:main} with initial datum $u_0\in L^{2}(\Omega,\Fcal_0,\P;L^{2}(\Ocal))$ satisfies
\[u\in L^{m+1}(\Omega\times [0,T];W^{1,m+1}_0(\Ocal))\cap L^\infty([0,T];L^2(\Omega;L^2(\Ocal))).\]
\end{theorem}
Note that also
\begin{equation*}u\in  C([0,T]; L^{m+1}(\Omega;L^{m+1}(\Ocal)))\cap L^2(\Omega;C([0,T],H))
\end{equation*}by the results of Gess and R\"ockner \cite{GR:2015}.

Our main idea is based on the observation that formally,
\begin{equation}\label{eq:chainrule}\Delta \left( u^{[m]}\right)=\div\left(\nabla \left( u^{[m]}\right)\right)=m\div\left(|u|^{m-1}\nabla u\right),\end{equation}
so the nonlinear drift is of the form $u\mapsto\div(A(u)\nabla u)$, that is, a divergence-form quasi-linear partial differential operator. The structure of the drift operator resembles the quasi-linear operators in \cite{DdMH:15,OW:19}, however, we would like to point out that their result requires strong ellipticity of the nonlinear coefficient $A$, whereas $A(u)=m |u|^{m-1}$ becomes singular for $u=0$ in our case. We will justify the formal chain rule \eqref{eq:chainrule} by a choice of suitable approximations for the nonlinearity $u^{[m]}$.

For the degenerate drift case, there are several strong regularity results for the stochastic porous medium equation, that is \eqref{eq:main} with $m>1$, as in this case one can treat the second order terms occurring in It\^o's formula more directly, see \cite{G:12}. Recently, optimal regularity for the stochastic porous medium equation in one spatial dimension with multiplicative space-time white noise was obtained using the so-called Stroock-Varopoulos inequality by Dareiotis, Gerencs\'er and Gess \cite{DGG:21}, see \cite{G:21,GST:20,HT:23} for further results proving improved regularity for porous media equations. We note that the application of the Stroock-Varopoulos inequality requires $m>1$ and cannot be applied in our case.

Furthermore, we would like to point out that our upper estimate contains a factor $m^{-\frac{m+1}{2}}$, so our argument does not to carry over to the singular multi-valued case $m=0$. Looking closer at our proof below, one observes that for $m=0$, we may obtain an upper bound containing a term $\frac{1}{\delta}\|u\|^3_{L^3(\Ocal)}$, with $\delta\to 0$, so even the improved integrability results from \cite{BR:09} in spatial dimensions $d=1,2,3$ cannot resolve this issue. A result of improved regularity in this limiting case remains an open problem.

\section{Proof of the main result}

Let us introduce a regularization for the nonlinearity $r\mapsto r^{[m]}$, for $\delta\ge 0$, let
\[\phi_\delta(r):=(r^2+\delta)^{\frac{m-1}{2}}r,\]
where $\phi_\delta \in C^1(\R)$ for $\delta>0$ with derivative
\[\phi'_\delta (r)=(r^2+\delta)^{\frac{m-3}{2}}(\delta+mr^2)\ge 0.\]

Now, we need to regularize the original equation \eqref{eq:main} with parameter $\eps>0$, $\delta\ge 0$
\begin{equation}\label{eq:eps}
\left\{
\begin{aligned}
du_{\eps,\delta}(t)&=\Delta \left( \phi_\delta(u_{\eps,\delta}(t))+\eps u_{\eps,\delta}(t)\right)\, dt+B\left(
u_{\eps,\delta}(t)\right)\, dW(t), &&  t\in (0,T],\quad\text{in}\;\Ocal,  \\ 
u_{\eps,\delta}(t)&=0, && t\in (0,T],\quad\text{on}\;\partial\Ocal,   \\ 
u_{\eps,\delta}(0) &=u_0, && \text{in}\;\Ocal,
\end{aligned}\right.
\end{equation}
Here, $\{W(t)\}_{t\ge 0}$ denotes the cylindrical Wiener process in $H$ on $(\Omega,\Fcal,\{\Fcal_t\}_{t\ge 0},\P)$, constructed with respect to the 
$\{\beta_k\}_{k\in\N}$ introduced in the description of the equation \eqref{eq:main} and the orthonormal basis $\{e_k\}_{k\in\N}$ of $H$.

By \cite[Proof of Theorem 3.1]{GR:2015},
see also \cite[Theorem 6.4]{GT:16}, we get that there exist unique solutions $u_{\eps,\delta}$ to \eqref{eq:eps} for any $\eps>0$, $\delta\ge 0$, and we obtain the following weak convergences (weak$^\ast$ convergence, respectively) for a subsequence $\{\delta_n\}_{n\in\N}$, $\lim_{n\to\infty}\delta_n=0$,
\begin{equation}\label{eq:conv}\begin{aligned}&u_{\eps,\delta_n}\rightharpoonup u_\eps&&\text{in $L^2(\Omega\times [0,T];H_0^1(\Ocal))$ as }&& n\to\infty,\\
&u_{\eps,\delta_n}\rightharpoonup^\ast u_\eps&&\text{in $L^2(\Omega; L^\infty([0,T];L^2(\Ocal)))$ as }&& n\to\infty,\\
&\Delta \phi_{\delta_n}(u_{\eps,\delta_n})\rightharpoonup \Delta u_\eps^{[m]}&&\text{in $L^2(\Omega\times [0,T];H)$ as }&& n\to\infty,\\
&u_{\eps}\rightharpoonup u &&\text{in $L^2(\Omega;C([0,T];H))$ as }&&\eps\searrow 0,
\end{aligned}\end{equation}
where $u$ is the unique solution to \eqref{eq:main} in the sense of \cite[Definition A.1]{GR:2015}.
On the other hand, $u_{\eps,\delta}\in L^2(\Omega\times [0,T];H_0^1(\Ocal))$, for any $\eps>0,\delta\ge 0$.

\begin{proof}[Proof of Theorem \ref{thm:main}]
We recall that we have assumed \eqref{eq:coeff} to hold.
Note that by the chain rule for Sobolev functions, as $\phi_\delta\in C^1(\R)$ for $\delta>0$ (composing $\phi_\delta$ with a smooth cut-off function if necessary), we get that for all $v\in H^1_0(\Ocal)$,
\[\nabla(\phi_\delta(v))=\phi'_\delta(v)\nabla v.\]
In the sequel, let us fix $t\in [0,T]$.
By It\^o's formula \cite[Theorem 4.32]{DPZ:14} for the functional
\[(v,t)\mapsto\|v\|^2_2\, e^{-Kt},\quad v\in L^2(\Ocal),\quad t\in [0,T],\]
for some $K\ge 0$, and by integration by parts in $\Ocal$,
we get that
\begin{eqnarray*}
\E\| u_{\eps,\delta}(t)\| _{2}^{2} \, e^{-tK}&= &\E\|u_0\|_2^2
-2\E\int_{0}^{t}\int_{\mathcal{O}}e^{-Ks}(\phi_\delta'(u_{\eps,\delta})+\eps)(\nabla u_{\eps,\delta}\cdot\nabla u_{\eps,\delta})\, d\xi \,ds \\
&&+\sum_{k=1}^\infty\E\int_0^t \int_\Ocal e^{-Ks} |g_k u_{\eps,\delta}|^2\,d\xi\,ds-K\E\int_0^t\int_\Ocal e^{-Ks}|u_{\eps,\delta}|^2\,d\xi\,ds\\
&\leq &\E\|u_0\|_2^2
-2\E\int_{0}^{t}\int_{\mathcal{O}}e^{-Ks}(\phi_\delta'(u_{\eps,\delta})+\eps)|\nabla u_{\eps,\delta}|^2\, d\xi \,ds \\
&&+(C_g-K)\E\int_0^t\int_\Ocal e^{-Ks} |u_{\eps,\delta}|^2\,d\xi\,ds
\end{eqnarray*}
For notation purposes, we have dropped the dependency on the time $s\in [0,t]$ and the space variable $\xi\in\Ocal$ for the functions $u_{\eps,\delta}$ under the integrals. This convention will be kept throughout the arguments below.

Choosing $K= C_g$, we obtain for $\eps,\delta\in (0,1]$,
\begin{equation}\label{eq:apriori}\begin{split}
&\E\| u_{\eps,\delta}(t)\| _{2}^{2} +2\E\int_{0}^{t}\int_{\mathcal{O}}(\phi_\delta'(u_{\eps,\delta})+\eps)|\nabla u_{\eps,\delta}|^2\, d\xi \,ds\\
\leq &\E\|u_0\|_2^2 \, e^{C_g t}.
\end{split}\end{equation}
We shall use this estimate to get the regularity of the
solution as follows.
We first rewrite for $\beta>0$,
\begin{equation*}
\E\int_{0}^{t}\int_{\mathcal{O}}| \nabla u_{\eps,\delta}| ^{m+1}\,d\xi\,
ds=\E\int_{0}^{t}\int_{\mathcal{O}}(\phi'_\delta(u_{\eps,\delta}))^{\beta }| \nabla u_{\eps,\delta}|
^{m+1}(\phi'_\delta(u_{\eps,\delta}))^{-\beta}\,d\xi\, ds,
\end{equation*}%
and use H\"{o}lder's inequality for $p=\dfrac{2}{m+1}$ and $q=\dfrac{2}{%
1-m}$.

We obtain 
\begin{eqnarray*}
&&\E\int_{0}^{t}\int_{\mathcal{O}}(\phi'_\delta(u_{\eps,\delta}))^{\beta }| \nabla u_{\eps,\delta}|
^{m+1}(\phi'_\delta(u_{\eps,\delta}))^{-\beta}\,d\xi\, ds\\
&\le&\left( \E\int_{0}^{t}\int_{\mathcal{O}}\left((\phi'_\delta(u_{\eps,\delta}))^{\beta}| \nabla
u_{\eps,\delta}| ^{m+1}\right) ^{\frac{2}{m+1}}\,d\xi\, ds\right) ^{\frac{m+1}{2}%
}\left( \E\int_{0}^{t}\int_{\mathcal{O}}(\phi'_\delta(u_{\eps,\delta}))^{-\beta \frac{2}{1-m}}\,d\xi\, ds\right)
^{\frac{1-m}{2}} \\
&=&\left( \E\int_{0}^{t}\int_{\mathcal{O}}(\phi'_\delta(u_{\eps,\delta}))^{\frac{2\beta }{m+1}}|
\nabla u_{\eps,\delta}| ^{2}\,d\xi \, ds\right) ^{\frac{m+1}{2}}\left(
\E\int_{0}^{t}\int_{\mathcal{O}}(\phi'_\delta(u_{\eps,\delta}))^{-\frac{2\beta }{1-m}}\,d\xi \,ds\right) ^{\frac{%
1-m}{2}}.
\end{eqnarray*}
If we choose $\beta =\frac{ m+1 }{2}$, we get
from the previous computations that for $\eps,\delta\in (0,1]$,
\begin{equation}\label{eq:finalbound}
\begin{split}
&\E\int_{0}^{t}\int_{\mathcal{O}}| \nabla u_{\eps,\delta}| ^{m+1}\,d\xi\,
ds\\
\leq & \left( \E\int_{0}^{t}\int_{\mathcal{O}}\phi'_\delta(u_{\eps,\delta})| \nabla
u_{\eps,\delta}| ^{2}\,d\xi\, ds\right) ^{\frac{m+1}{2}}\left( \E\int_{0}^{t}\int_{%
\mathcal{O}}(\phi'_\delta (u_{\eps,\delta}))^{-\frac{m+1}{1-m}}\,d\xi \,ds\right) ^{\frac{1-m}{2}}\\
= & \left( \E\int_{0}^{t}\int_{\mathcal{O}}\phi'_\delta(u_{\eps,\delta})| \nabla
u_{\eps,\delta}| ^{2}\,d\xi\, ds\right)^{\frac{m+1}{2}}\\
&\quad\times\left( \E\int_{0}^{t}\int_{%
\mathcal{O}}(u_{\eps,\delta}^2+\delta)^{\frac{(m-3)(m+1)}{2(m-1)}}(\delta+m u_{\eps,\delta}^2)^{-\frac{m+1}{1-m}}\,d\xi \,ds\right) ^{\frac{1-m}{2}}\\
\le & \left( \E\int_{0}^{t}\int_{\mathcal{O}}\phi'_\delta(u_{\eps,\delta})| \nabla
u_{\eps,\delta}| ^{2}\,d\xi\, ds\right)^{\frac{m+1}{2}}\\
&\quad\times\left( \E\int_{0}^{t}\int_{%
\mathcal{O}}(u_{\eps,\delta}^2+\delta)^{\frac{(m-3)(m+1)}{2(m-1)}}(m\delta+m u_{\eps,\delta}^2)^{-\frac{m+1}{1-m}}\,d\xi \,ds\right) ^{\frac{1-m}{2}}\\
\leq & \left( \E\int_{0}^{t}\int_{\mathcal{O}}\phi'_\delta(u_{\eps,\delta})| \nabla
u_{\eps,\delta}| ^{2}\,d\xi\, ds\right)^{\frac{m+1}{2}}\\
&\quad\times C(m)\left(\E\int_{0}^{t}\int_{%
\mathcal{O}}(u_{\eps,\delta}^2+\delta)^{\frac{(m-3)(m+1)+2m+2}{2(m-1)}}\,d\xi \,ds\right)^{\frac{1-m}{2}}\\
\leq & \left( \E\int_{0}^{t}\int_{\mathcal{O}}\phi'_\delta(u_{\eps,\delta})| \nabla
u_{\eps,\delta}| ^{2}\,d\xi\, ds\right)^{\frac{m+1}{2}}\\
&\quad\times C(m)\left(\E\int_{0}^{t}\int_{%
\mathcal{O}}(u_{\eps,\delta}^2+\delta)^{\frac{m+1}{2}}\,d\xi \,ds\right)^{\frac{1-m}{2}}\\
\leq & \left( \E\int_{0}^{t}\int_{\mathcal{O}}\phi'_\delta(u_{\eps,\delta})| \nabla
u_{\eps,\delta}| ^{2}\,d\xi\, ds\right)^{\frac{m+1}{2}}\\
&\quad\times C(m)\left( \E\int_{0}^{t}\int_{%
\mathcal{O}}\left(|u_{\eps,\delta}|^{m+1}+\delta^{\frac{m+1}{2}}\right)\,d\xi \,ds\right)^{\frac{1-m}{2}},
\end{split}
\end{equation}
where $C(m):=m^{-\frac{m+1}{2}}$.
The first factor is bounded by \eqref{eq:apriori} and the second factor is bounded by \eqref{eq:Lm1}, where the bounds do not depend on $\eps,\delta\in (0,1]$, compare with \cite[Theorem 3.1, Lemma 3.3 and the respective proofs]{GR:2015}. By \eqref{eq:apriori} or \eqref{eq:conv}, we know that
\begin{equation}\label{eq:epsbound}
u_{\eps,\delta_n}\in L^{2}(\Omega\times [0,T];H_0^1(\Ocal)),\quad\eps>0,\delta\ge0.
\end{equation}
By \eqref{eq:conv}, and the fact that
\[L^{2}(\Omega\times [0,T];H_0^1(\Ocal))\subset L^{m+1}(\Omega\times [0,T];W^{1,m+1}_0(\Ocal))=:X,\]
we may take the limit $n\to\infty$, and, in virtue of the uniform estimate in $\delta$ from \cite[Theorem 3.1]{GR:2015}, obtain 
that \eqref{eq:finalbound} holds for $u_\eps$, $\eps\in (0,1]$.
Also, as the bounds used above are independent of $\eps\in (0,1]$, so we get that the family $\{u_{\eps}\}_{\eps>0}$ is uniformly bounded in the reflexive Banach space $X$,
where the Sobolev trace is seen to be zero on $\partial\Ocal$ $\P\otimes dt$-a.e. by \eqref{eq:epsbound}.
Thus, we can extract a weakly convergent subsequence $\{u_{\eps_k}\}_{k\in\N}$, $\lim_{k\to\infty}\eps_k=0$ with a weak limit $\tilde{u}\in X$.
By the weak convergence \eqref{eq:conv} $u_{\eps}\rightharpoonup u$ as $\eps\searrow 0$ in the Hilbert space $L^2(\Omega\times [0,T];H)$, we obtain easily by duality arguments that $u=\tilde{u}$ $\P\otimes dt$-a.e. in $W^{1,m+1}_0(\Ocal)$.

The uniform bound of $\{u_{\eps,\delta}\}_{\eps,\delta\in (0,1]}$ in $L^\infty([0,T];L^2(\Omega;L^2(\Ocal)))$ follows from \eqref{eq:apriori}, where we can obtain $u\in L^\infty([0,T];L^2(\Omega;L^2(\Ocal)))$ by similar weak convergence arguments.
\end{proof}

\end{document}